\documentclass[10pt,a4paper]{article}

\usepackage{amssymb,amsbsy,amsmath,amsfonts,amscd,mathrsfs}
\usepackage{amsthm}
\usepackage{graphicx} 
\usepackage[dvipsnames]{xcolor}
\usepackage{a4wide}
\usepackage{enumerate}
\usepackage{subfigure}

\usepackage[ruled,vlined]{algorithm2e}

\newtheorem{theorem}{Theorem}
\newtheorem{corollary}{Corollary}
\newtheorem{lemma}{Lemma}
\newtheorem{definition}{Definition}
\newtheorem{remark}{Remark}
\newtheorem{proposition}{Proposition}

\newtheorem{example}{Example}

\newtheorem{opq}{Open Question}
\newtheorem{assumption}{Assumption}
\providecommand{\rmj[1]}{{\color{black}#1}}

\newcommand{\rd}{{\mathbb R^d}}
\newcommand{\re}{{\mathbb R}}
\newcommand{\rhoprime}{{{}stabilization radius}}

\newcommand{\n}{{\mathbb N}}
\newcommand{\robar}{{\tilde \rho}}

\newcommand{\cM}{{\cal{M}}}
\newcommand{\cm}{{\cal{M}}}

\newcommand{\inter}{{{\rm int} }}

\newcommand{\setmat}{{\cal{M}}}
\newcommand{\z}{{\mathbb{Z}}}
\newcommand{\cN}{{\mathbb{N}}}

\newcommand{\vectun}{{e }}
\providecommand{\com[2]}{\begin{tt}[#1: #2]\end{tt}}

\providecommand{\rmjj[1]}{{\color{black}#1}}

\providecommand{\rmjjj[1]}{{\color{black}#1}}
\providecommand{\ad[1]}{{\color{black}#1}}
\newcommand{\pao}[1]{{\color{black}#1}}

\newcommand{\colp}[1]{{\color{black}#1}}

\title{\LARGE{\bf On feedback stabilization of linear switched systems via switching signal control}}
\author{Rapha\"el M. Jungers\footnote{R.~Jungers  is with the ICTEAM Institute, Universit\'e catholique de Louvain, 
    4 avenue Georges Lemaitre, B-1348 Louvain-la-Neuve, Belgium. R. J. is an F.R.S.-FNRS research associate. His work is supported by the Belgian Network DYSCO, funded by the Belgian government and the Concerted Research Action (ARC) of the French Community
of Belgium. {\tt\small raphael.jungers@uclouvain.be}}
, Paolo Mason\footnote{P.~Mason is with Laboratoire des Signaux et Syst\`emes (L2S, UMR 8506), CNRS -
CentraleSup\'elec - Universit\'e Paris-Sud, 3, rue Joliot Curie,
91192, Gif-sur-Yvette, France. His work is supported by  the iCODE
institute, research project of the Idex Paris-Saclay. {\tt\small paolo.mason@l2s.centralesupelec.fr}}}%

\begin{document}

\maketitle

\begin{abstract}
Motivated by recent applications in control theory, we study the feedback stabilizability of switched systems, where one is allowed to chose the switching signal as a function of $x(t)$ in order to stabilize the system.  We propose new algorithms
and analyze several mathematical features of the problem which were unnoticed up to now, to our knowledge.  We prove complexity results, \colp{(in-)equivalence between various notions of stabilizability,} existence of Lyapunov functions, and provide a case study for \colp{a paradigmatic example introduced by Stanford and Urbano}.
\end{abstract}

\section{Introduction} \label{sec-intro}

Switched systems are a paradigmatic family of complex systems, which has sparked many interesting research efforts in the last decades.  They appear naturally in many engineering situations, or as abstractions of more complicated systems.  Also, the mathematical structure of the challenges they offer turns out to appear in other fields of mathematics, engineering, or computer science, not only in control theory.  Already the simplest questions which one might ask on a dynamical system turn out to be extremely challenging in this framework.\\
 Let us mention one question: {Is it possible to design a switching signal giving rise to an asymptotically stable behaviour, i.e. such that for any starting point the corresponding trajectory converges to zero?}  It is well known that the answer to this question for discrete time linear switched systems depends on the so-called \emph{joint spectral subradius}, and this quantity is known to be Turing-uncomputable \pao{\cite{blondel-mortal},\cite[Section 2.2.4]{jungers_lncis}}.  Other works studied more advanced problems in this `open loop' framework.  See for instance \cite{lincoln2002lqr,lee2009infinite} for works on the LQR problem.\\
In this paper, we study a similar question, though bearing some differences: \colp{Under which conditions there exists} 
a switching sequence that drives the point towards the origin \emph{if one is allowed to observe the system at every time, and choose the switching signal accordingly}?  
Thus, if the joint spectral subradius characterizes the open loop stabilizability question \colp{assuming the switching signal independent of the initial state}, the quantity we focus  \colp{on here is instead related} 
to the feedback stabilization problem: suppose that one can observe at every time $t$ the value of $x(t),$ and is allowed to control the system \colp{based on this sole information; }
 is there a strategy allowing to stabilize the system, for any starting point?  \\This question appears naturally in several applications, and has motivated the interest of control theorists in recent years.  However, to our knowledge, if several techniques have been proposed, which allow in some cases to obtain sufficient conditions for stabilizability, it has not received a thorough mathematical analysis, like some other control theoretic questions on switched systems (see, e.g., \cite{jungers_lncis,liberzon-switching}).  

{The question is of application in problems of considerable importance nowadays.  See for instance \cite{lee-dul2011} for a study of this problem motivated by supervisory control and measurement scheduling.  Another application is the problem of optimizing drug treatments of some viral pathologies, like HIV: it has been recently proposed to model the medication process of HIV by a switched system \cite{hmcb10}, and the question of stabilizing it through the use of different modes comes naturally into the picture. Other important applications are in bisimulation of hybrid systems \cite{girard2010approximately}, where the switching system is an abstraction of a more complex one.  See \cite{buisson2005stabilisation,jadbabaie2003coordination,liberzon1999basic,bolzern2015h} for more works and applications around the stabilization of switching systems.}

Before to survey the previous results on the topic, let us formalize our problem: Given a set of matrices $\mathcal{M}\subset\mathbb{R}^{d\times d},$ we analyze linear dynamical systems where the matrix can switch within the set $\mathcal{M}.$ 
We denote the solution starting from $x_0\in\mathbb{R}^d$ and associated with a switching sequence $\sigma:\mathbb{T}\to \cM$ as $x_{\sigma,x_0}(t)$, where $\mathbb{T}=\mathbb{N}$ or $\mathbb{T}=\mathbb{R}_+$ represents the (discrete or continuous) time. \\In the first case $x_{\sigma,x_0}(t)=A_{\sigma(t)}\dots A_{\sigma(1)}x_0$ is the solution of the difference equation \[x(s+1)=A_{\sigma(s)}x(s), \quad x(0)=x_0,\] 
while in the second case $x_{\sigma,x_0}(\cdot)=\Phi_{\sigma}(s)x_0$ is the solution of \[\dot x(s)=A_{\sigma(s)}x(s), \quad x(0)=x_0,\] 
and $\Phi_{\sigma}(s)$ is the associated fundamental matrix. \rmjjj{\pao{Notice that in the expressions above,} 
by a slight abuse of notation, we  
\pao{have denoted} by $A_{\sigma(t)}$ the matrix from $\cM$ given by the signal $\sigma$ at time $t.$} \pao{This notation will be employed again in the following.}

Several different controllability or stabilization problems can be posed on switched systems\rmjjj{. \pao{In the presence of additional control variables, an} 
interesting problem concerns the feedback stabilization under arbitrary switching.}
In another family of problems, one does not consider control variables, but tries to stabilize the system via the choice of the switching signal. Of course these two approaches can be mingled, and one can seek for a controller, together with a switching strategy, so that the \emph{joint action of the linear controller, and the controller-ruled switching signal} stabilizes the system (see \cite{bolzern2015h,zhang2009exponential,lee2009infinite}).  Our goal is to understand from a theoretical point of view these different features in a separate analysis, and, as said above, we focus here on the second question.  Thus, even though very interesting pieces of work exist on the (linear) controller design problem for switched systems (see, e.g., \cite{blanchini,leed06,liberzon-switching,lin-antsaklis}), we focus here on works interested in the second problem.  

Feedback stabilization for continuous-time systems has attracted a considerable attention in the (non-linear) control community since many decades.
In particular let us mention that many noteworthy results have been obtained starting from the 80s, dealing with necessary and/or sufficient conditions for asymptotic controllability and stabilizability as well as existence and properties of control-Lyapunov functions, see for instance~\cite{ancona1999patchy,artstein,brockett,clarke,freeman,sontag2,sontag}. In many cases these results apply to the switched systems setting. On the other hand there exist only few works focusing on stabilizability for the specific case of (linear) continuous-time switched systems. In this context let us mention that for two dimensional continuous-time switched systems the stabilizability property can be characterized in a quite explicit way (see e.g. \cite{cong}), namely reducing the stabilizability property to the verification of a few simple algebraic inequalities similarly to what was done for the uniform stability issue (see e.g \cite{balde,boscain}).

Concerning discrete-time switched systems, to our knowledge the first work mentioning the advantage of making use of feedback in order to design a stabilizing switching signal is \cite{stanford}, where situations like the one described in Example \ref{ex-urbano} below are presented.  Subsequent papers have provided algorithmic ways of constructing such a switching rule: In \cite{zhai2001quadratic,zhai2003quadratic}, a first sufficient LMI (Linear Matrix Inequalities) condition is given.  This condition ensures the existence of an ellipsoid, which can be left invariant under a proper state feedback switching rule.  Characterizations of the existence of such an invariant ellipsoid (often termed as \emph{quadratic stability}), are given in \cite{skafidas1999stability,wicks1998switched}.  In \cite{geromel2006stability}, a more general condition is given, which allows for more complex invariant sets, described by piecewise quadratic functions.  This condition (under its most general form) is a BMI (Bilinear Matrix Inequality), but relaxations of it are given, which can be more easily checked, and still are more general than the above mentioned condition.  Other conditions can be found in \cite{fiacchini-jungers,geromel2006stability,pettersson2003synthesis}.  Interesting improvements on the BMI approach have been obtained recently \cite{fiacchini-girard-jungers}. Among other results, the authors provide a geometric interpretation and generalize the approach in order to decrease conservatism.

A further improvement has been obtained in \cite{zhang2009exponential}, where it is shown how to iterate such conditions in order to decrease their conservativeness, at the cost of increasing the computation time (this paper actually tackles both the control via the switching signal and via control inputs).  Asymptotically, these conditions become necessary; that is, if the system is stabilizable via control of the switching signal, there is a condition in the hierarchy that will be satisfied.  See \cite{lee-dul2011} for similar results with a slightly different notion of stabilizability.  As pointed out by the authors, in practice, for reasonably large systems, the conditions could become too computationally expensive before to reach one that is satisfied.  This is not surprising in view of the difficulty of the problem, which we formally prove in Section \ref{sec-algo}.  

In fact, already for very small dimensional systems, the problem seems very hard to tackle.  In order to support this claim, we study in details such a system in Section~\ref{stanford}.
\colp{For this system, based on the theoretical results obtained in the preceding sections, w}e propose a few approaches in order to approximate \colp{the best convergence rate that one can ensure.  Then we present a technical result providing an obstruction to the computation of the optimal achievable rate of decay, which we leave as an open problem.
} 

{\bf Outline of the paper.}  In Section \ref{sec-properties}, we adopt a systematic approach and provide careful definitions and elementary mathematical properties for our problem \colp{in the discrete-time setting}. In Section \ref{sec-control}, we make use of two tools in the control theory literature that can be put to good use in our situation: control-Lyapunov functions, and the joint spectral subradius. In Section \ref{sec-algo}, we analyze the algorithmic side of the problem.  Unsurprisingly, many negative results can be derived: we show that the general stabilizability problem is undecidable, and that it is NP-hard, even for nonnegative matrices.  We provide two algorithms that respectively deliver a lower and an upper bound on the required quantity (to our knowledge this is the first systematic lower bound in the literature). 
In Section~\ref{stanford} we
revisit an old example from Stanford and Urbano, and propose it as a paradigmatic example of the complexity of the problem.  We provide a detailed (numerical and theoretical) analysis of this problem. Still, a complete understanding of the stabilization properties of the system is missing.  \colp{In Section~\ref{s-cont} we derive the continuous-time counterparts of the results obtained in the preceding sections.}
We finally raise two interesting (in our opinon) open questions that we have not been able to answer.

\begin{remark}
As said above, our paper aims at showing several of the many subtleties underlying the stabilizability problem for switched systems.  As such, it contains results that are sometimes not connected to each other.  Also, as we will see, many different modeling details (discrete or continuous time, control signal defined a priori or depending on the state,...) give rise to different notions. A thorough analysis of all the properties for all the possible models would go out of the scope of this paper.  Rather, we have selected several interesting (in our view) results that give a flavor of the richness of the topic.  \\
Also, in this paper, we will make the following standing assumption.

\begin{assumption}\label{assum}
$\mathcal{M}$ is a compact set of matrices. If $\mathbb{T}=\mathbb{R}_+$ we additionally assume that $\mathcal{M}$ is convex.
\end{assumption}
\end{remark}

\section{Definitions and elementary properties}\label{sec-properties}
In this section we formalize some basic notions of stabilizability we will deal with and some related elementary properties. For the sake of clarity, starting from this section we  focus on the discrete-time setting (that is, $\mathbb{T}=\mathbb{N}$), \colp{and we assume, according to Assumption~\ref{assum}, that $\cM$ is a compact set of matrices. Nevertheless,} many of the definitions and results below may be straightforwardly adapted to the continuous-time case. The latter case will be discussed at the end of this paper, 
in Section~\ref{s-cont}.

In the definition below we introduce two natural notions of stabilizability related to the use of open loop  and closed loop switching laws, respectively.
\begin{definition}\label{def-pointwise}
We say that $\mathcal{M}$ is \emph{pointwise stabilizable} if for any $x_0\in \mathbb{R}^d$ there exists a switching law $\sigma_{x_0}(\cdot)$ such that $\lim_{t\to\infty} x_{\sigma_{x_0},x_0}(t)=0$.\\
We say that $\mathcal{M}$ is \emph{feedback stabilizable} if there exists a $0$-homogeneous\footnote{A function $f:\ \rd\rightarrow \re$ is \emph{$0$-homogeneous} if $f(\lambda x)=f(x)\quad \forall \lambda\neq 0,\ \forall x.$ } function $\sigma: \ \re^d\rightarrow \cM$ such that, for any $x_0\in\re^d$, $\lim_{t\to\infty} x(t)=0$, where $x(\cdot)$ is a solution of the corresponding closed loop system starting at $x_0$. In addition we say that $\mathcal{M}$ is uniformly feedback stabilizable if there exists $\sigma(\cdot)$ as above such that for any positive $\epsilon_1,\epsilon_2$ there exists $T\in \cN$  such that if $|x_0|<\epsilon_1$ the corresponding trajectory $x(\cdot)$ satisfies $|x(t)|<\epsilon_2$  if $t\geq T$.
\end{definition}

The difference between the definitions of pointwise and feedback stabilizability
is subtle. In the former pointwise stabilizability, one has to decide the switching signal once and for all at time zero. For general control systems, the 
notion of pointwise stabilizability is essentially known as asymptotic controllability (see e.g.~\cite{sontag-book}). In the latter feedback stabilizability, we assume that the controller has the knowledge of the state $x(t)$ at every instant, and the controller has to decide its actuation based on this sole knowledge. 


In order to tackle the more general situation where one is interested in the minimal achievable rate of growth of the system, we now introduce numerical values that quantify the notion of stabilizability.
\begin{definition}
We introduce the following stabilizability indices:
\begin{itemize}
\item For a given $x\in\mathbb{R}^d$ let us consider the quantity
\[\tilde{\rho}_x(\cM)=\inf\{\lambda\geq 0\ \bigl| \  \exists \sigma(\cdot)\mbox{ switching sequence, }M>0\ \mbox{s.t. } |x_{\sigma,x}(t)|\leq M \lambda^t |x|, \forall t\geq 0\}
\]
and define the \emph{pointwise stabilization radius} of the switched system as
\[\tilde{\rho}(\cM)=\sup_{x\in\mathbb{R}^d} \tilde{\rho}_x(\cM).\]
\item In the previous definition the constant $M$ may depend on the initial condition. This is no more the case if we consider the following definition:
\[\tilde{\rho}'(\cM)=\inf\left\{\lambda\geq 0 \ \Bigl| \ 
\begin{array}{c}
\exists M>0 \mbox{ s.t. } |x_{\sigma,x}(t)|\leq M \lambda^t |x| \mbox{ for any }x\in\mathbb{R}^d,\\
t\geq 0\mbox{ and some switching signal }\sigma \mbox{ depending on }x
\end{array}
\right\}.\] 
\end{itemize}
\label{ref-def}
\end{definition}
A closely related quantity, that has received more attention in the literature, is the \emph{joint spectral subradius} 
\[\check{\rho}(\cM):=\lim_{t\to \infty} \inf_{\sigma(\cdot)} \|A_{\sigma(t)}\dots A_{\sigma(0)}\|^{1/t}.\]  In terms of control theory, this latter quantity represents the 
{open loop stabilizability} of a switched system by a switching law which is independent on the initial point, unlike the ones involved in the computation of $\tilde \rho(\cM)$.
See \cite{jungers_lncis} for more on the joint spectral subradius.
To the best of our knowledge, Stanford and Urbano \cite{stanford} are the first ones to observe that $\check \rho(\cM) \neq \tilde \rho(\cM).$ They provide an example for which the two quantities are different, which we reproduce in a slightly different form here below. See \cite{bochi2015continuity} for a recent study of a very similar example (though not concerned with feedback stabilizability).

\begin{example}\label{ex-urbano} 
Let us consider $\mathcal{M}=\{A_1,A_2\}$, where
\[A_1=\left(\begin{array}{cc}\cos\frac{\pi}4 & \sin\frac{\pi}4 \\ -\sin\frac{\pi}4 & \cos\frac{\pi}4 \end{array}\right)=\frac{\sqrt{2}}{2}\left(\begin{array}{cc}1&1\\-1&1\end{array}\right), \quad A_2=\left(\begin{array}{cc} \frac12 & 0 \\ 0 & 2  \end{array}\right).\]
In this case it is easy to see that $\check{\rho}(\cM)=1.$ Indeed, $\det (A_{\sigma(t)}\dots A_{\sigma(0)})=1$ independently on the switching sequence, which implies that $\|A_{\sigma(t)}\dots A_{\sigma(0)}\|\geq 1$ and thus $\check{\rho}(\cM)\geq 1.$ On the other hand we have $\check{\rho}(\cM)\leq \|A_1\|=1$ by submultiplicativity, and thus  $\check{\rho}(\cM)=1.$\\
If one is allowed to tune the switching sequence depending on the value of $x(t),$ the situation is different: indeed, for any $x\in \mathbb{R}^2$ there always exists a natural number $n_x\leq 3$ such that the absolute value of the angle formed by the vector $A_1^{n_x}x$ and the $x_1$ axis  is smaller or equal than $\pi/8$. As a consequence it is easy to obtain the estimate $|A_2A_1^{n_x}x|<0.9 |x|$ and thus $\tilde\rho(\cM)<0.9^{1/4}\sim 0.974$.
For similar examples and a more detailed exposition, see e.g.~\cite{stanford}.
\end{example}
In the next lemma, we state properties that we will need in the rest of the paper. 
\begin{lemma}
\label{lem-basic}
The pointwise stabilization radius satisfies  the following basic properties:
\begin{itemize}
\item[(i)] Homogeneity: For any compact set of matrices $\cM,$ $\forall \gamma >0,$ $\tilde \rho (\gamma \cM) = \gamma \tilde \rho (\cM)$,
\item[(ii)] For any compact set of matrices $\cM,$ $\forall t \in \n,$ $\tilde \rho (\cM^t) = \tilde \rho (\cM)^t.$
\end{itemize}
The same properties hold for $\tilde \rho'$.
\end{lemma}
\begin{proof}
It is easy to see that there is a bijection between trajectories of the system defined by $\cM$ and the one defined by $\gamma \cM$ (resp. $\cM^t$) such that the number $\lambda$ in the definition of $\tilde\rho$  and $\tilde\rho'$ becomes $\gamma \lambda$ (resp.  $\lambda ^t$). 
The thesis follows immediately.
\end{proof}

We now establish the equivalence of the two quantities introduced in Definition~\ref{ref-def} and we characterize  the notion of pointwise stabilizability in terms of them. 
The following proposition may be easily obtained by adapting the proof of \colp{the similar result~\cite[Theorem~3.9]{sun2006switched}.}
\begin{proposition}
\label{equiv}
The following conditions are equivalent:
\begin{itemize}
\item[(i)] $\tilde{\rho}'(\cM)<1$,
\item[(ii)] $\tilde{\rho}(\cM)<1$,
\item[(iii)] $\mathcal{M}$ is pointwise stabilizable.
\end{itemize}
\end{proposition}
\begin{remark}
By Lemma~\ref{lem-basic} and Proposition~\ref{equiv} we deduce that $\tilde{\rho}(\cM)$ and $\tilde{\rho}'(\cM)$ coincide.
\end{remark}
We now present a rather surprising result. 
\rmj{The intuitive meaning of the quantity $\tilde{\rho}$ is the following: for an arbitrary switching system, for any initial condition $x_0,$ one can find a switching signal $\sigma$ such that \begin{equation}\label{eq-feedbackvsopen}\limsup_{t\to\infty} |x_{\sigma,x_0}(t)|^{1/t}\leq \tilde{\rho}(\cM).\end{equation}  
One may \colp{ask whether such a} 
switching signal \colp{may be} 
given by some $0$-homogeneous function $\sigma(x),$ as a feedback control function.  Surprisingly,} \colp{this is not always the case. Roughly speaking, in general the ``most stabilizing'' switching behavior is not attainable by a $0$-homogeneous feedback.} 
\begin{proposition}
\label{prop-different}
For any $x_0\in \re^d,$ there exists a switching law $\sigma_{x_0}(\cdot)$ such that 
\begin{equation}\limsup_{t\to\infty} |x_{\sigma_{x_0},x_0}(t)|^{1/t}\leq \tilde{\rho}(\cM).\label{eq-prop}\end{equation}
On the other hand there does not always exist 
a $0$-homogeneous feedback such that each  solution $x(\cdot)$ of the corresponding closed loop system satisfies $\limsup_{t\to\infty} |x(t)|^{1/t}\leq \tilde{\rho}(\cM)$.
\end{proposition}
\begin{proof}
Let us prove the first part of the statement. First, if $\tilde{\rho}_{x_0}(\cM)<\tilde{\rho}(\cM),$ the proof is obvious.  So, let us suppose that $\tilde{\rho}_{x_0}(\cM)=\tilde{\rho}(\cM).$ 
We construct a suitable solution $x_{\sigma_{x_0},x_0}(\cdot)$ of the system such that $\limsup_{t\to\infty} |x_{\sigma_{x_0},x_0}(t)|^{1/t}=\tilde{\rho}(\cM)$. Let us denote by $M_n$ the value of $M$ in the definition of $\tilde{\rho}'$ corresponding to $\lambda=\tilde{\rho}(\cM)+2^{-n-1}$.
By definition of $\tilde{\rho}'$ we know that  it is possible to construct $x_{\sigma_{x_0},x_{0}}(\cdot)$ such that there exists $t_1\in \mathbb{N}$ large enough satisfying $|x_{\sigma_{x_0},x_{0}}(t_1)|\leq \frac1{M_1} (\tilde{\rho}(\cM)+1)^{t_1} |x_{0}|$. Define $x_1=x_{\sigma_{x_0},x_{0}}(t_1)$. Reasoning as before we can \pao{prolong $x_{\sigma_{x_0},x_{0}}(\cdot)$  on $[0,t_1+t_2]$, 
where} $t_2\in \mathbb{N}$ is such that $|x_{\sigma_{x_0},x_{0}}(t_1+t_2)|\leq \frac{M_1}{M_2} (\tilde{\rho}(\cM)+\frac12)^{t_2} |x_{1}|\leq \frac{1}{M_2} (\tilde{\rho}(\cM)+\frac12)^{t_2}(\tilde{\rho}(\cM)+1)^{t_1} |x_{0}|$ and in addition $|x_{\sigma_{x_0},x_{0}}(t_1+s)| \leq  (\tilde{\rho}(\cM)+\frac12)^{s}(\tilde{\rho}(\cM)+1)^{t_1} |x_{0}|$ for $s\in [0,t_2]$. Following the same lines one can prolong indefinitely the trajectory adding at each time an interval of length $t_n\in \mathbb{N}$ to its domain in such a way that
\[
|x_{\sigma_{x_0},x_{0}}(t)| \leq (\tilde{\rho}(\cM)+2^{-N})^{t-T_N} \prod_{n=1}^{N} (\tilde{\rho}(\cM)+2^{-n+1})^{t_n}|x_0|,
\]
where $T_N=\sum_{n=1}^{N} t_n$ and  $t\in [T_N,T_{N+1}]$. Without loss of generality we assume that  $T_N$ goes to infinity with $N$, so that $t$ can be arbitrarily large.
We have
\[
\log (|x_{\sigma_{x_0},x_{0}}(t)|^{1/t}) \leq \frac{t-T_N}t \log(\tilde{\rho}(\cM)+2^{-N}) +\sum_{n=1}^N\frac{t_n}t \log(\tilde{\rho}(\cM)+2^{-n+1})+\frac1t \log(|x_0|).
\]
which converges to $\log(\tilde{\rho}(\cM))$ as $t$ goes to infinity, concluding the proof of the first part.

We now provide an explicit example to prove the  second part of the proposition.
Let us consider the discrete-time system corresponding to the following set of matrices
\begin{equation}
\cM=\{A,B\}=\left\{\begin{pmatrix}2&0&0\\0 & 1&0\\0&0&1\end{pmatrix},\begin{pmatrix}0&1&1\\0&1&1\\0&1&1\end{pmatrix}\right \}.
\end{equation}
We claim that for any $0$-homogeneous function $\sigma(x)$  the corresponding feedback solution $x_{\vectun}(\cdot)$ starting at the all-ones vector $\vectun$ (i.e., $\vectun = (1,1,1)^T$) is such that  $\limsup_{t\to\infty} |x_{\vectun}(t)|^{1/t}> 1$. We proceed by contradiction: let us suppose that there exists such a function $\sigma(x)$ with $\limsup_{t\to\infty} |x_{\vectun}(t)|^{1/t}\leq 1$.
First, certainly $A_{\sigma(\vectun)}\neq B,$ because $B\vectun = 2\vectun,$ and this would imply that $|x_{\vectun}(t)|^{1/t}= 2.$  Also, certainly this trajectory 
must contain some multiplication by $B,$ because $|A^t\vectun|\approx 2^t.$  Let us call $T$ the first $t$ such that $x(t)$ is multiplied by $B.$ Since $BA^T\vectun = 2\vectun,$ the trajectory of $x$ must be of the shape $\dots BA^TBA^T\vectun,$ and we see that $\lim_{t\to\infty} |x_{\vectun}(t)|^{1/t}>1,$ leading to a contradiction.\\
Let us now show that $\tilde{\rho}_{\vectun}(\cM)=\tilde{\rho}(\cM)=1$.
Since $\|BA^t\|$ is bounded by a constant which is independent on $t,$ and obviously $\tilde\rho(\cM^{t+1})\leq \|BA^t\|,$ it turns out by item~$(ii)$ in Lemma~\ref{lem-basic} that $\tilde{\rho}(\cM)\leq 1$. 
The inequality $\tilde{\rho}_{\vectun}(\cM)\geq 1$ comes from the fact that the second component is always non-decreasing along trajectories of the system starting at $\vectun$.
\end{proof}

We conclude this section with a result showing a sort of optimality property of the \rhoprime. 
\begin{proposition}
\label{p-optim}
\colp{There always exist $x_0\in \mathbb{R}^d$ and $\sigma_0(\cdot)$ such that  $|x_{\sigma_0,x_0}(t)|\leq \tilde{\rho}(\cM)^t |x_0|$ for any $t\in \mathbb{N}$.}
\end{proposition}
\begin{proof}
For an arbitrary $\epsilon>0,$ let $\lambda_{\epsilon}=\tilde{\rho}(\cM)+\epsilon$ and consider a trajectory $x_{\sigma_{\epsilon},x_{\epsilon}}(\cdot)$ satisfying $|x_{\sigma_{\epsilon},x_{\epsilon}}(t)|\leq M\lambda_{\epsilon}^t |x_{\epsilon}|$ for any $t\in \mathbb{N}$ and some $M>0$. Let $t_{\epsilon}$ such that
\begin{equation}
\sup_{t\in \mathbb{N}}\frac{|x_{\sigma_{\epsilon},x_{\epsilon}}(t)|}{\lambda_{\epsilon}^t |x_{\epsilon}|} \leq (1+\epsilon)\frac{|x_{\sigma_{\epsilon},x_{\epsilon}}(t_{\epsilon})|}{\lambda_{\epsilon}^{t_{\epsilon}} |x_{\epsilon}|}.\label{sup1}
\end{equation}
Define $\hat{x}_{\epsilon}=\frac{x_{\sigma_{\epsilon},x_{\epsilon}}(t_{\epsilon})}{|x_{\sigma_{\epsilon},x_{\epsilon}}(t_{\epsilon})|}$ and $\hat{\sigma}_{\epsilon}(\cdot)=\sigma_{\epsilon}(t_{\epsilon}+\cdot)$. Then from \eqref{sup1} we have 
\begin{equation}\label{sup2}
\sup_{t\in \mathbb{N}}\frac{|x_{\hat{\sigma}_{\epsilon},\hat{x}_{\epsilon}}(t)|}{\lambda_{\epsilon}^t} \leq 1+\epsilon.
\end{equation}
Passing to the limit as $\epsilon$ goes to $0$, we have (up to subsequences) that  $\hat{x}_{\epsilon}$ converges to some unit vector $x_0$ and $x_{\hat{\sigma}_{\epsilon},\hat{x}_{\epsilon}}$ converges to a trajectory $x_{\sigma_0,x_0}$ uniformly on compact subsets of $\mathbb{N}$. Thus,  passing to the limit at the left- and right-hand side of \eqref{sup2} we get the thesis.
\end{proof}

\section{Two tools from control theory}\label{sec-control}
\colp{In this section we analyze the stabilizability notions introduced above in the discrete-time setting by taking advantage of two classical tools from control theory: control-Lyapunov functions and the joint spectral subradius.}

\subsection{Control-Lyapunov functions and their mathematical properties}
\label{sec-lyap}
%
Let us introduce the following adapted definition of control-Lyapunov function.
\begin{definition}\label{Lyap-def}
We say that a positively homogeneous\footnote{A function $f:\ \re^d\rightarrow \re$ is said \emph{positively homogeneous} if $f(\lambda x)=|\lambda| f(x)\quad \forall \lambda\neq 0,\ \forall x.$ }  function $V:\mathbb{R}^d\to \mathbb{R}_+$ is a control-Lyapunov function for $\mathcal{M}$ if 
\begin{itemize}
\item[(1)] there exist two constants $0<m\leq M$ such that $m|x|\leq V(x) \leq M|x|$, 
\item[(2)] there exists $\mu\in(0,1)$ such that for any $x\in\mathbb{R}^d$ there exists $\sigma_x(\cdot)$ with $V(x_{\sigma_x,x}(t))<\mu^t V(x)$ for $t\in \mathbb{N}$. 
\end{itemize}
\end{definition}
If there exists a control-Lyapunov function for $\mathcal{M}$ and  $\mu$ is given by item $(2)$ above, then clearly $\tilde\rho(\cM)\leq \mu<1$, that is the system is pointwise stabilizable. If in addition the control-Lyapunov function is regular enough, then the system turns out to be (uniformly) feedback stabilizable, as stated by the result here below.
\begin{proposition}\label{Lyap}
If  $\mathcal{M}$ admits a Lipschitz continuous control-Lyapunov function $V$ then $\mathcal{M}$ is uniformly feedback stabilizable, and the stabilizing feedback $\sigma(\cdot)$ can be taken piecewise constant. 
\end{proposition}
\begin{proof}
The existence of a stabilizing feedback is trivial, as it is enough to associate with any $x\in \mathbb{R}^d$ \colp{such that $|x|=1$ the value $\sigma_x(0)$, where $\sigma_x$ is as in Definition~\ref{Lyap-def}, and then extends by homogeneity on the whole space $\mathbb{R}^d\setminus\{0\}$}. 

The fact that the stabilizing feedback can be taken piecewise constant follows from the continuity of $V$ and the continuous dependence of the solutions with respect to the initial data. 
Indeed these properties imply that a switching law that produces a certain  rate of decrease of $V$ at time $1$ starting  at $x_0$ provides the same  rate of decrease starting from any point of  a small enough neighborhood of $x_0$, which can be extended by homogeneity to a cone. By a simple compactness argument there exists a finite number of cones covering $\mathbb{R}^d\setminus\{0\}$ in each of which the stabilizing feedback law may be taken constant. 
This concludes the proof of the proposition.
\end{proof} 
We now provide a converse Lyapunov theorem for our class of systems. We skip the proof, since our result is a quite straightforward adaptation of the one presented in \cite[Section 4.3.1]{sun}.   
\begin{proposition}
\label{Vlam}
For any $\lambda>\tilde{\rho}(\cM)$ the function $V_{\lambda}:\mathbb{R}^d\to \mathbb{R}_+$ 
\begin{equation}
\label{lyap-eq}
V_{\lambda}(x)=\sup_{t\geq 0}\inf_{\sigma(\cdot)}\frac{|x_{\sigma,x}(t)|}{\lambda^t}
\end{equation}
is well-defined, absolutely homogeneous, Lipschitz continuous and satisfies $V_{\lambda}(x_{\sigma,x}(t))\leq\lambda^t V_{\lambda}(x)$ for any $x\in\mathbb{R}^d,\ t\in\mathbb{N}$ and some $\sigma(\cdot)$, depending on $x$. In particular $\mathcal{M}$ is pointwise stabilizable if and only if it admits an absolutely homogeneous control-Lyapunov function.
\end{proposition}

An immediate consequence of the above result is the equivalence between the pointwise stabilizability and the uniform feedback stabilizability as stated in the following corollary. Notice that this result does not contradict Proposition~\ref{prop-different}, \rmj{even though they are opposite in spirit: the corollary below indicates that feedback stabilizability and pointwise stabilizability are equivalent, even though, for feedback stabilizability, the ``most stabilizing'' switching behavior is not always attainable by a homogeneous feedback.}

\begin{corollary}\label{cor-feedback-pointwise}
The three equivalent conditions in Proposition~\ref{equiv} are also equivalent to the following:
\begin{itemize}
\item[(iv)] $\mathcal{M}$ is uniformly feedback stabilizable.
\end{itemize}
\end{corollary}
\begin{proof}
Clearly, if $\mathcal{M}$ is uniformly feedback stabilizable then it is also pointwise stabilizable.
The other implication trivially follows from Proposition~\ref{Lyap} and Proposition~\ref{Vlam}. 
\end{proof}

\vspace{5pt}

\begin{remark}
$V_{\lambda}(x)$ might not be convex.  This makes $\tilde \rho$ hard to compute. For instance, it hampers our ability to provide a guaranteed accuracy for our algorithms, while this is possible for, for instance, the joint spectral radius (see \cite[Theorem 2.12]{jungers_lncis}).  Indeed, the joint spectral radius admits a similar notion of Lyapunov function, but in that case it is convex.
\end{remark}

\begin{remark}
\label{Vhat}
The function 
\[\hat{V}_{\lambda}(x)=\inf_{\sigma(\cdot)}\sup_{t\geq 0}\frac{|x_{\sigma,x}(t)|}{\lambda^t}\]
is also a control-Lyapunov function for $\mathcal{M}$. This function can be shown to be lower semi-continuous  but in general there is no clear indication that it is also Lipschitz. 
\end{remark}


\subsection{The joint spectral subradius}
As already observed (see Example~\ref{ex-urbano}), in general the joint spectral subradius $\check\rho(\cM)$ can be strictly larger than the stabilization radius. 
We now show that if all the matrices in $\cM$ have only positive entries, then $\tilde\rho(\cM)=\check \rho (\cM).$\\
We say that a matrix $A$ possesses an invariant (proper) cone $K$ if $ A
K\subset K .$  We say that a convex closed cone $K'$ is {\em embedded in} $K$ if
$(K'\setminus \{0\}) \subset  \inter K.$ In this case we call $\{K, K'\}$ an
{\em embedded pair}. (The embedded
cone $K'$ may be degenerate, i.e., may have an empty interior.) An embedded pair $\{K,K'  \}$ is called
an {\em invariant pair} for a matrix
$A$ (resp. a set of matrices $\cM$) if the cones $K$ and  $K'$ are
both invariant for $A$ (resp. for all the matrices in $\cM$).  Finally, given a particular cone $K,$ for two vectors $x,y\in \re^d,$ we note $x\geq_K y$ for $x-y\in K.$
{ Given $v\in \re^d\setminus\{0\}$ we define $\re^d_v$ as the half-space $\{x\in\re^d | v^Tx\geq 0\}$. We say that a cone is \emph{proper} if it is contained in some half-space.
\begin{lemma}
\label{l-convexcone}
If a convex closed cone $K$ is proper, then there exists $C>0$ such that $w_1\geq_K w_2\geq_K 0$ implies $|w_1|\geq C|w_2|$.
\end{lemma}
\begin{proof}
By contradiction assume there exist two sequences $\{w^k_1\}_{k\geq1},\{w^k_2\}_{k\geq1}$ of vectors in $K$ such that, for $k\geq1$, $w_1^k\geq_K w_2^k\geq_K 0$, $|w_2^k|=1$ and $\lim_{k\to \infty} w_1^k=0$. Without loss of generality we assume that $w_2^k$ converges to a unit vector $w^*\in K$. Then one has that $w_1^k- w_2^k\in K$ converges to $-w^*\in K,$ which contradicts the fact that $K$ is embedded in $\re^d_v$.
\end{proof}
}
\begin{theorem}
If the matrices in $\cM$ have a common invariant embedded pair of proper cones {$\{K,K'\},$ then $\tilde\rho(\cM)=\check \rho (\cM).$}
\end{theorem}
\begin{proof}
We make use of results from \cite{protasov-jungers-blondel09} for approximating the joint spectral subradius.  In that paper, it is shown that for any set of matrices $\cM$ with an invariant cone $K,$ the following quantity is a lower bound for $\check \rho(\cM):$

\begin{equation}\label{eq-sigma} \check\sigma_K (\cM) =  \sup \Bigl\{ \lambda  \geq
 0 \quad
\bigl| \quad \exists \ v  \geq_K 0 ,  v \ne 0\quad A v \geq_K \lambda
 v \quad \forall  A  \in  \cM \Bigr\}.\end{equation}

It turns out that $\check\sigma_K(\cM)$ is actually also a valid lower bound for $\tilde\rho(\cM).$
 Indeed, it is easy to see that if $ A v \geq_K \lambda v ,$ then $ BA v \geq_K \lambda
 Bv, $ provided that $B$ leaves $K$ invariant.  Thus, the existence of $v$ as in (\ref{eq-sigma}) for each $\lambda<\check\sigma_K(\cM)$ implies that $$Av \geq_K \lambda^t v,\quad \forall A\in \cM^t,$$ and thus, by Lemma~\ref{l-convexcone}, $|Av| \geq C\lambda^t |v|\quad \forall A\in \cM^t$, which implies $\robar(\cM) \geq \check\sigma_K(\cM)$. Similarly, one has that $\robar(\cM)\geq \check\sigma_K(\cM^t)^{1/t}.$

Now, if the set $\cM$ has an embedded pair of invariant cones, we have that (\cite[Corollary 3]{protasov-jungers-blondel09}) $$ \lim_{t\rightarrow \infty}\check\sigma_K(\cm^t)^{1/t}=\check\rho(\cM).$$ Combining the above equations, one gets $\robar(\cM)=\check\rho(\cM).$ \end{proof}

\begin{corollary}
If all the matrices in $\cM$ have only positive entries, then $\tilde\rho(\cM)=\check \rho (\cM).$
\end{corollary}
\begin{proof}
As shown in \cite[Corollary 1]{protasov-jungers-blondel09}, sets of positive matrices share an invariant embedded pair of \pao{proper} cones.
\end{proof}

Thus, in the case where the matrices share an embedded pair of cones, one can use algorithms for providing a lower bound to the joint spectral subradius $\check \rho(\cM)$ in order to obtain a lower bound on $\tilde \rho(\cM).$  See for example \cite{protasov-jungers-blondel09,GP11} for methods that perform well in practice (even sometimes terminate in finite time).  

\section{Algorithms}
\label{sec-algo}
We explore here the algorithmic computation of $\tilde \rho$ and focus on the discrete-time setting.
\subsection{Complexity}

Let us show that one should not expect a polynomial time algorithm for the problem.  We add that the proofs in this subsection use the same technique as in \cite{jungers-traj}.
\begin{proposition}\label{prop-np}
Unless $P=NP,$ there is no polynomial time algorithm for solving the \pao{stabilizability} problem, even for nonnegative matrices.  More precisely, there is no algorithm that receives a pair of nonnegative integer matrices $\cM$, and that decides in polynomial time (w.r.t. the size and number of matrices) whether $\tilde\rho(\cM)<1.$
\end{proposition}
\begin{proof}
Our proof is by reduction from the {\textit{mortality problem}} which is known to be NP-hard, even for
nonnegative integer matrices \cite[p. 286]{blondel-mortal}.  In this problem, one is given a set of
matrices $\setmat,$ and it is asked whether there exists a product of matrices in \pao{$\setmat$} 
which is equal to the zero matrix.

We now construct an instance $\cM'$ such that $\tilde\rho (\cM')<1$ if and only if the set $\cM$ is mortal: simply take $\cM'=\{A'=2A\ |\ A\in
\cM\}.$ 

Suppose first that $\cM$ is mortal.  Then, the corresponding product in $\cM'$ is also equal to zero, and thus $\tilde \rho (\cM')=0.$ \\
If, on the other hand, $\cM$ is not mortal, then
\pao{all products of matrices in $\setmat$ have nonnegative integer entries, so that we have}
$$ \forall A \in \cM^t,\quad |A \vectun|\geq 1, $$ and thus,
$$ \forall A' \in \cM'^t,\quad |A' \vectun|\geq 2^t $$
(we recall that we note $\vectun=(1,\dots,1)^T$).  This implies that $\tilde \rho (\cM') \geq 2.$
\end{proof}

If one relaxes the requirement that the matrices and the vectors are nonnegative, then the problem
becomes even harder, and can be undecidable, as shown in the next proposition.
\begin{proposition}
The \pao{stabilizability} problem is undecidable.
\end{proposition}
\begin{proof}
It is known that the mortality problem with matrices having entries in $\z$ is undecidable \cite[Corollary
2.1]{jungers_lncis}.  We reduce this problem to the \pao{stabilizability} problem in a way similar as in Proposition
\ref{prop-np}, except that we build larger matrices of size $n^2$ (where $n$ is the size of the initial matrices).  The matrices in the set $\cM'$ are of the shape $\cM'=\{\mbox{diag}(2A,\dots,2A)\ |\ A \in \cM\},$ with $n$ copies of the same matrix on the diagonal. 

Now if $\cM$ is mortal, again $\tilde \rho = 0.$\\
If, on the other hand, $\cM$ is not mortal, then every product $A$ of matrices in $\cM$ has a nonzero column (say, the $i$th one), and thus, the product $A\vectun_i$ has a norm larger than one, where $\vectun_i$ is the $i$th standard basis vector (recall that the entries have integer values).  Thus, defining \pao{$v=(e_1,\dots, e_n)$} (that is, the concatenation of the $n$ standard basis vectors), for any product $A'$ in $\cM'^t,$ we have that $|A'v|\geq 2^t,$ and $\tilde\rho(\cM')\geq2.$
\end{proof}
\subsection{Upper bound}
In this section we provide an algorithm, which not only derives successive upper bounds on $\tilde \rho,$ but also has the property that these bounds asymptotically converge to the true value $\tilde \rho.$ Note that algorithms with such a property were recently proposed in \cite{zhang2009exponential,lee-dul2011}.   \rmj{Our algorithm is more elementary, and this might be a good property for scaling issues.  However, each iteration is very simple, and it is expected that for practical applications, where only a few iterations are done, other methods, relying on powerful optimization technologies like Semidefinite Programming, will outperform the technique developed here.}  We then describe how to adapt the algorithm in order to make it even more scalable, at the price of loosing the guarantee of convergence. Our algorithm is based on the following \rmj{simple proposition:
\begin{proposition}
If $\mathcal{M}$ is pointwise stabilizable then for any $\varepsilon>0$ there exists a large enough time  $t_\varepsilon$ such that  every point of the unit ball can be mapped at time $t_\varepsilon$, with a suitable feedback switching strategy, into the ball centered at the origin and of radius $\varepsilon$. 
\end{proposition}}
\begin{proof}
Since $\tilde{\rho}'<1$ there exists $\lambda<1$ and $M>0$ such that  $|x_{\sigma,x}(t)|<M \lambda^t |x|$ for any $x\in \mathbb{R}^d$ and some $\sigma(\cdot)$ depending a priori on $x$. It is thus enough to take $t\geq \frac{\ln (\varepsilon/M)}{\ln \lambda}$.
\end{proof}
This proposition allows us to propose an asymptotically tight upper bound on the \rhoprime.  In other words, we provide a sequence of sufficient conditions for stabilizability, which are asymptotically necessary.  So, the theoretical efficiency of our method is similar to the one from \cite{zhang2009exponential} (see the introduction for more details). We think that the simplicity of the method (compute products of increasing length t and iterate) offers a valuable alternative to \cite{zhang2009exponential}, which uses more complex optimization methods.  Also, we think that this approach bears advantages in comparison with \cite{geromel2006stability} in that one has a clear procedure to follow, which is guaranteed to terminate after a finite amount of time if the system is stabilizable.  On the contrary, the matrix inequalities proposed in \cite{geromel2006stability} are non-convex, and thus there is no clear procedure to efficiently find the optimal upper bound provided by these inequalities.  On top of that, in \cite{geromel2006stability}, if the optimal upper bound is found, one has no guarantee that this bound is close to the actual value of $\tilde \rho,$ while our Algorithm \ref{algo-upperbound} is guaranteed to converge asymptotically.  \rmj{However, as said above, for a fixed number of iterations (or, say, when computational time is a strong constraint), it is expectable that the previous algorithms, which are based on more involved optimization techniques, outperform the one presented here.}

\begin{algorithm}\label{algo-upperbound}

\KwData{A set of matrices $\cM$ and an integer $T$}
\KwResult{Outputs a sequence of upper bounds $r_t,$ $t=1,\dots, T$\;
\Begin{

\nl $t:=0$\;
\nl \While{$t\leq T$} {
\nl $t:=t+1$\;
\nl Compute $\cM^t$ \;
\nl Minimize $\gamma$ such that
\begin{equation}\label{probl-upperbound} \forall x : |x|=1, \, \exists A \in \cM^t\ s.t. |Ax|\leq \gamma\end{equation}
\nl Output $r_t:=\gamma$\;}}}
\caption{An asymptotically tight algorithm}
\end{algorithm}

\begin{theorem}\label{thm-algo-upperbound}
Algorithm \ref{algo-upperbound} iteratively provides an upper bound $r_t\geq \tilde \rho$ which asymptotically converges towards $\tilde \rho.$ 
\end{theorem}
\begin{proof}
At every step of the algorithm, one must solve the optimization problem \ref{probl-upperbound}, which can be solved by bisection on $\gamma.$  Indeed, for a fixed $\gamma,$ the (in)equalities in Problem \eqref{probl-upperbound} are algebraic, and can be solved with standard procedures, like Tarski-Seidenberg elimination \cite{Tarski_quantifier_elim,Seidenberg_quantifier_elim}
\end{proof}
\begin{remark}
Theorem \ref{thm-algo-upperbound} provides a theoretical procedure that asymptotically converges towards the true value.  It is well known that algebraic equation solvers are very slow in practice, so, practical implementations of algorithm \ref{thm-algo-upperbound} would advantageously make use of heuristics in order to approximately solve \eqref{probl-upperbound}.
\end{remark}
We conclude this section by establishing a result that will be applied in the next section and which \colp{provides an alternative way to estimate the stabilization radius.}
\begin{proposition}\label{upper-improve}
Let $\mu>0$ and $\bar t\in \mathbb{N}$. Suppose that for every $z\in \mathbb{R}^2$ of norm 1 there exists a product $A_z\in \cM^{t_z}$ for some $t_z\leq \bar t$ such that $|A_z z|^{1/t_z}\leq \mu,$ then, $$\tilde{\rho}(\cM)\leq \mu.$$\end{proposition} 
\begin{proof}Let us extend $A_z,t_z$ by homogeneity: $A_z:=A_{z/|z|},t_z=t_{z/|z|}$. Then for any $z_0\in \mathbb{R}^2$ we can construct a trajectory $z(\cdot)$ by concatenation as follows 
$\dots A^{(k)}A^{(k-1)}\dots A^{(1)}z_0$, where we define recursively starting from $i=0$ the matrices $A^{(i+1)}=A_{z_i}$
, with $z_i=A^{(i)}\dots A^{(1)}z_0$. In particular $|z_i|\leq \mu^T |z_0|$ where $T=\sum_{j=0}^{i-1} t_{z_j}$ is the total time employed to attain $z_i$ from $z_0$. By compactness of $\cM$ we have that $\|M\|$ is uniformly bounded on $\cup_{t\leq \bar t} \cM^t$ by a constant $C>0$, from which we deduce that $|z(t)|\leq C\max\{1,\mu^{-\bar t}\}\mu^t |z_0|$. The thesis follows. 
\end{proof}

\subsection{Lower bound}
\label{s-lb}
In the previous section we have presented an algorithm providing an upper bound on $\tilde\rho(\cM)$. The lower bound appears to be much more complicated to analyze numerically, and the purpose of  this section is to provide an analytic criterion to estimate it. This is, to our knowledge, the first algorithm providing a lower bound on $\tilde \rho(\cM)$ in the literature.

\begin{theorem}\label{lem-sing}
One has $\tilde\rho(\cM)\geq \min_{A\in\mathcal{M}}\sigma_{m}(A)$ where $\sigma_{m}(A)$ is the smallest singular value of the matrix $A$.
\end{theorem}
\begin{proof}
Recall that $\sigma_{m}(A)$, the square root of the smallest  eigenvalue of the non-negative definite symmetric matrix $A^TA$, satisfies 
\[\sigma_{m}(A)=\min_{x\in\mathbb{R}^d\setminus\{0\}}\frac{|Ax|}{|x|}. \]
We thus have 
\[\frac{|A_{\sigma(t)}\dots A_{\sigma(1)}x|}{|x|} = \frac{|A_{\sigma(t)}\dots A_{\sigma(1)}x|}{|A_{\sigma(t-1)}\dots A_{\sigma(1)}x|}\dots  \frac{|A_{\sigma(1)}x|}{|x|}\geq \left(\min_{A\in\mathcal{M}}\sigma_{m}(A)\right)^t\]
for any $\sigma(\cdot)$.
Thus, by definition of $\tilde{\rho}'(\cM)=\tilde{\rho}(\cM)$ for any $\lambda>\tilde{\rho}(\cM)$ there exists $M>0$ such that $\frac{\lambda^t}{M}> \left(\min_{A\in\mathcal{M}}\sigma_{m}(A)\right)^t$ for every $t>0$, that is $\lambda\geq \min_{A\in\mathcal{M}}\sigma_{m}(A)$, proving the thesis. 
\end{proof}
One can iterate this lower bound for longer and longer products of matrices in $\cM,$ thanks to Lemma \ref{lem-basic}.  It can be seen on simple examples (even on a single matrix) that this can improve the lower bound. 
Unfortunately,  
the method may also fail to converge to the true value of $\tilde\rho(\cM).$ \pao{Indeed, in the proof of the theorem above, we could actually replace $\tilde{\rho}(\cM)$ by $\tilde{\rho}_x(\cM)$. In other words, $\min_{x\in\mathbb{R}^d\setminus\{0\}}\tilde\rho_x(\cM)\geq \min_{A\in\mathcal{M}}\sigma_{m}(A)$, but the left-hand side of the previous equality coincides with $\tilde\rho(\cM)$ only if $\tilde\rho_x(\cM)$ is constant outside the origin. This is not the case for instance if $\cM$ is given by a single matrix with at least two eigenvalues with different modulus.}

\section{A case study: the Stanford--Urbano example}
\label{stanford}
In this section we return to  Example \ref{ex-urbano}, and apply our results.  We show that already for this simple example, many open questions remain.

We will see in particular that although numerical algorithms allow to compute upper bounds for the stabilizability radius, its computation is complicated by the fact that  lower bounds appear to be much more difficult to obtain. Indeed,  a simple application of \pao{Theorem}~\ref{lem-sing} to the matrices from Example \ref{ex-urbano} gives $\tilde \rho(\cM) \geq 1/2$, and applying the same \pao{result to $\cM^t$, together with Item~$(ii)$ in Lemma~\ref{lem-basic},}  does not improve the bound. Worse than that, we will show that $\tilde\rho_z(\cM)=\frac12$ for $z$ belonging to a dense subset of $\mathbb{R}^2$, which suggests that improvements of the trivial lower bound could possibly be obtained only with much more subtle techniques.

Before reworking Example~\ref{ex-urbano}, let us consider the even simpler example given by 
\[\bar{\mathcal{M}}=\{\bar{A}_1,\bar{A}_2\}\quad \mbox{with }\bar{A}_1=A_1^2=\left(\begin{array}{cc} 0 & 1 \\ -1 & 0 \end{array}\right)\quad\mbox{ and }\quad\bar{A}_2=A_2.\]
This set of matrices has some interesting properties.  
\begin{proposition} 
The index $\tilde{\rho}_z(\bar\cM)$ is equal to one everywhere except on the principal axes, where $\tilde{\rho}_z(\bar\cM)=\frac12$.\end{proposition}
\begin{proof}One can check that $\bar{A}_1^4=\mathrm{Id}$ (so that $\bar{A}_1^{-1}=\bar{A}_1^3$), $\bar{A}_2\bar{A}_1\bar{A}_2=\bar{A}_1$ (in particular $\bar{A}_2^{-1}=\bar{A}_1^3\bar{A}_2\bar{A}_1$), and these conditions imply that any product in $\tilde{\cM}^t$ coincides with a product of the form $\bar{A}_1^{t_3}\bar{A}_2^{t_2}\bar{A}_1^{t_1}$  for some $ t_1,t_3\in \{0,1,2,3\}$ and $t_2\geq 0$ (and $t_1+t_2+t_3\leq t$). As a consequence of this fact it is easy to see that the corresponding switched system is not stabilizable, unless the initial point belongs to one of the principal axes. Indeed one has that  $|\bar{A}_1^{t_3}\bar{A}_2^{t_2}\bar{A}_1^{t_1}z_0|$ is either equal to $|\bar{A}_2^{t_2}z_0|$ or to $|\bar{A}_2^{t_2}\bar{A}_1z_0|$. If $z_0$ does not belong to the $x_1$ axis then the trajectory $t\mapsto \bar{A}_2^t z_0$ diverges and thus $|\bar{A}_2^{t_2}z_0|\geq C_0$ for any $t_2\geq 0$ where $C_0>0$ is a given constant depending on $z_0$. Similarly  if $z_0$ does not belong to the $x_2$ axis then the trajectory $t\mapsto \bar{A}_2^t \bar{A}_1z_0$ diverges and $|\bar{A}_2^{t_2}\bar{A}_1z_0|\geq C_0$ \pao{for any $t\geq 0$,} for some $C_0>0$.
\pao{We deduce that any trajectory starting from a point $z_0$ outside the principal axes is bounded  from below and does not converge to the origin. Since obviously $\tilde{\rho}(\bar\cM)\leq 1$, this concludes the proof of the proposition.}
\end{proof}

The system in Example \ref{ex-urbano} bears some similarities with the previous one although the dependence of $\tilde{\rho}_z(\cM)$ with respect to $z$ appears to be much more complicated and we will only be able to describe it partially. 
In order to derive a better upper bound on $\tilde\rho(\cM)$ than the one already obtained in Example~\ref{ex-urbano} we use Proposition~\ref{upper-improve}. In particular, by  taking  $\bar t=9$ as maximal value for the times $t_z$ in  Proposition~\ref{upper-improve} we now improve the upper bound to $\tilde \rho(\cM)\leq 0.886$. 
In order to select the matrices $A_z$ in an optimal way let us remark that, since $A_1$ is an isometry and our upper bound is necessarily smaller than one, we can restrict to matrix products whose last multiplication is made with the matrix $A_2$. Also, $A_1^4=-\mathrm{Id}$, and thus we may consider powers of $A_1$ with maximum exponent $3$. Moreover $A_2A_1^2 A_2=A_1^2$, and thus  powers of $A_1$ with exponent $2$ could be effective only at the beginning of the sequence.
Consider the following matrices:
\begin{eqnarray*}
&A[1]=A_2,\quad A[2]=A_2A_1,\quad A[3]=A_2A_1^2,\quad A[4]=A_2A_1^3,\quad A[5]=A_2A_1A_2A_1^2,&\\
&A[6]=A_2A_1A_2A_1A_2,\quad A[7]=A_2A_1A_2A_1A_2A_1,\quad A[8]=A_2A_1A_2A_1^3,&\\
&A[9]=A_2A_1^3A_2A_1^2,\quad A[10]=A_2A_1A_2A_1A_2A_1^2,\quad A[11]=A_2A_1^3A_2A_1^3,&\\
 &A[12]=A_2A_1A_2A_1A_2A_1^3,\quad A[13]=A_2A_1A_2A_1A_2A_1A_2A_1^2.&
\end{eqnarray*}
If $\alpha\in S^1$ represents the angle between the point $z$ and the $x_1$ axis, we can partition the half circle $[0,\pi]$ into intervals with each of which we associate a matrix product
as in Figure~\ref{fig-urbano}.
Of course, the above estimate of $\tilde \rho(\cM)$ can most probably be tightened by considering longer matrix products.

As an alternative method, estimates of the Lyapunov functions in Section~\ref{sec-lyap} can also be computed numerically. Figure~\ref{plateau}(a) shows an estimate of the candidate Lyapunov function described in Remark~\ref{Vhat} taking $\lambda=0.88$. This estimate is obtained by computing an approximation of $\hat{V}_{\lambda}$ on a large number of points and by interpolating such values. One can see that the corresponding ratio $\min\{\hat{V}_{\lambda}(A_1x),\hat{V}_{\lambda}(A_2x)\}/\hat{V}_{\lambda}(x)$, which represents the largest decrease rate of $\hat{V}_{\lambda}$ along trajectories and is depicted in Figure~\ref{plateau}(b), slightly exceeds the value $0.88$  only on very  small intervals. Such a numerical analysis suggests that  $\tilde \rho(\cM)<0.88$. For smaller values of $\lambda$ the function $\hat{V}_{\lambda}$ appears to be much more irregular and the condition $\min\{\hat{V}_{\lambda}(A_1x),\hat{V}_{\lambda}(A_2x)\}/\hat{V}_{\lambda}(x)\leq \lambda$ is violated on larger intervals, so that there is no clear indication that $\tilde \rho(\cM)<\lambda$.
\begin{figure}
\centering
\includegraphics[width=0.5\textwidth]{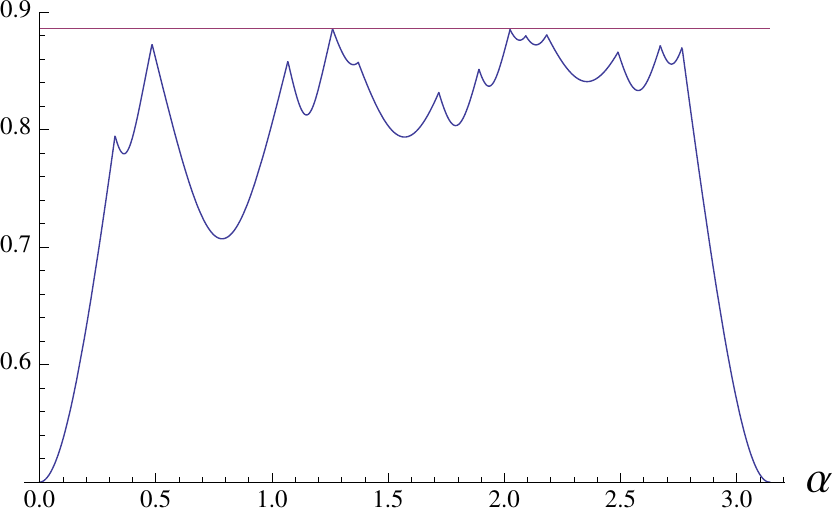}
\caption{The function $F(\alpha)=\min_i \big(|A[i]z_{\alpha}|^{1/{t_i}}\big)$, where $z_{\alpha}=(\cos\alpha,\sin\alpha)^T$ and $t_i$ is the length of the matrix product $A[i].$ Its maximum is an upper bound on the feedback stabilization radius. This maximum is approximately equal to $0.886$.}
\label{fig-urbano}
\end{figure}
\begin{figure}
\begin{center}
\subfigure[The level set $\hat{V}_{\lambda}^{-1}(1)$]{\includegraphics[width=.42\textwidth]{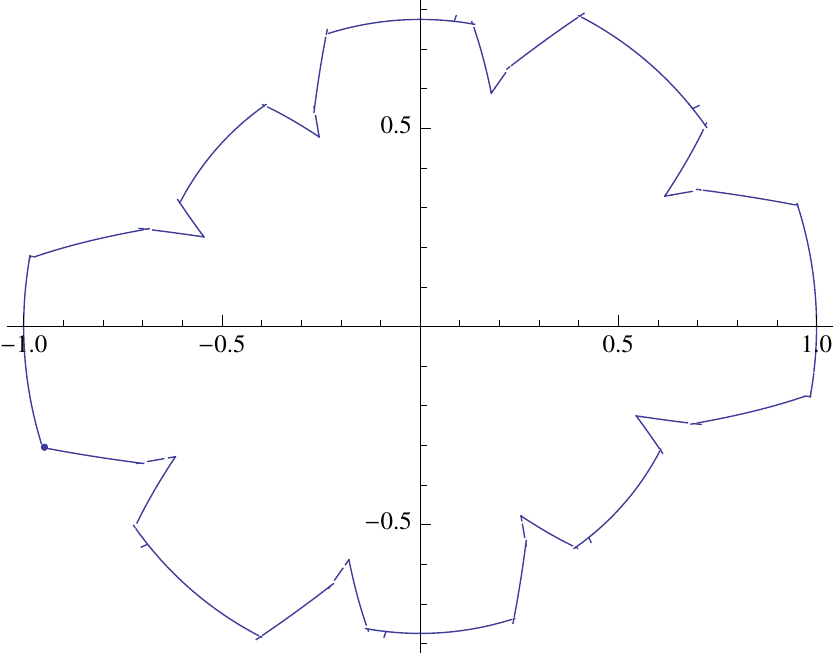}}\hspace{40pt}
\subfigure[Ratio $\min\{\hat{V}_{\lambda}(A_1x),\hat{V}_{\lambda}(A_2x)\}/\hat{V}_{\lambda}(x)$ for $x$ belonging to the first two quadrants.]{\includegraphics[width=.42\textwidth]{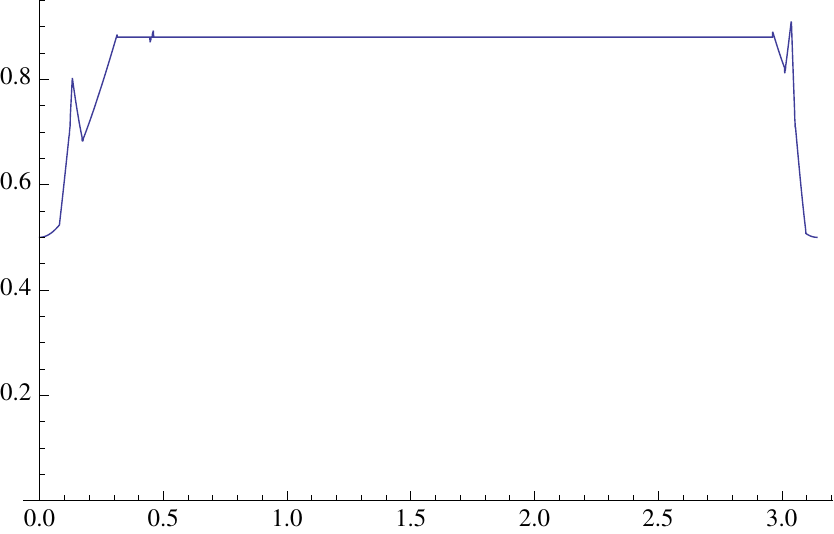}}
\end{center}
\caption{The function $\hat{V}_{\lambda}$ for $\lambda=0.88$.}
\label{plateau}
\end{figure}

\colp{In the following we will derive a result highlighting a technical obstruction to further improving the trivial lower bound $\tilde{\rho}(\cM)>\frac12$. In order to do so, we introduce the notion of forward orbit of a point $z_0$, defined as 
\[\{x_{\sigma,z_0}(t)\in\mathbb{R}^2\ |\ \sigma:\mathbb{N}\to \{1,2\},\ t\in\mathbb{N}\}.\]
For our particular matrices, it turns out that the forward orbit of $z_0$ 
actually coincides with the complete orbit of $z_0$, i.e. 
\[\{A_{\sigma(t)}^{\nu(t)}\dots A_{\sigma(0)}^{\nu(0)}z_0\ |\ \sigma:\mathbb{N}\to \{1,2\},\nu:\mathbb{N}\to \{-1,1\},\ t\in\mathbb{N}\}.\]
This is an immediate consequence of the fact that  $A_1^{-1}=A_1^7$ and $A_2^{-1}=A_1^2 A_2 A_1^{-2}=A_1^2 A_2 A_1^{6}$.}
Looking at the union of the orbits starting from points on the $x_1$ axis we can say something more, namely that it is symmetric with respect to the $x_1$ axis and coincides with the set of points from which there exists a trajectory reaching the $x_1$ axis in finite time:
\[\mathcal{R}=\left\{z\in \mathbb{R}^2\setminus\{0\}\ \Bigl|\ (1,0)^T=\frac{Az}{|Az|},\mbox{ for some } A\in \mathcal{M}^t,\ t\in\mathbb{N}\right\}\cup\{0\}.\]
\begin{proposition}
\label{prop-dense}
The set $\mathcal{R}$ is a countable union of straight lines passing through the origin. The corresponding angles with the $x_1$ axis form a strict subset of the set of angles having rational tangent and this subset is dense in $\re$. As a consequence $\mathcal{R}$ is dense in $\re^2$. 
\end{proposition}
\begin{proof}
The set $\mathcal{R}$ is made of a countable union of straight lines passing through the origin since it contains all those for which the tangent of the angle with the $x_1$ axis is $4^k,\ k\in\mathbb{Z}.$
Also, the tangent of the angle formed by each line in $\mathcal{R}$ with the $x_1$ axis is a rational number. Indeed the maps $A_1$ and $A_2$ act on the angles in such a way that a rational tangent $\frac{p}q$ becomes $\frac{4p}q$ and $\frac{p-q}{p+q}$, respectively, and the set $\mathcal{R}$ can be identified with the (integer) couples $\{p,q\}$ obtained by applying iteratively these two operations starting from $\{p,q\}=\{0,1\}$. We can also assume that at each iteration the numbers $p,q$ are positive (the orbit is symmetric with respect to the $x_1$ axis) and relatively prime. In particular, applying $A_1$ or $A_1^{-1}$ leads, up to the sign, to the couple $\{\frac{p+q}2,\frac{|p-q|}2\}$, when both $p,q$ are odd, while it leads to $\{p+q,|p-q|\}$ if one among $p,q$ is even. On the other hand, applying $k$ times $A_2$ or $A_2^{-1}$ leads to couples of the form $\{4^kp,q\}$ or $\{4^kq,p\}$, up to a common divisor. 

We claim that all the couples $\{p,q\}$ that can be generated with these operations starting from the $x_1$ axis are such that both $p,q$ are different from $2$ in $\mathbb{Z}/4\mathbb{Z}$.\\
Since applying $A_2$ or $A_2^{-1}$ to the $x_1$ axis keeps it invariant, we start by applying $A_1$ or $A_1^{-1}$, and we get the couple $\{1,1\}$.
We now apply $A_2^k$ or $A_2^{-k}$ and we get a couple which, in $\mathbb{Z}/4\mathbb{Z}$, is equal to $\{0,1\}$. We apply again $A_1$ or $A_1^{-1}$ and we get a couple of the form $\{\pm 1,\pm 1\}$ in $\mathbb{Z}/4\mathbb{Z}$. A further application of $A_2^k$ or $A_2^{-k}$ leads to $\{0,\pm 1\}$ in $\mathbb{Z}/4\mathbb{Z}.$ In particular we get that all the  couples that can be obtained by iterating this procedure must be of the form $\{0,\pm 1\}$ or $\{\pm 1,\pm 1\}$ in $\mathbb{Z}/4\mathbb{Z}$, proving the claim.

We deduce that $\mathcal{R}$ cannot be identified with the set of all angles having rational tangents. For instance it does not include the angle of tangent  $\frac12$, as well as all the elements of the corresponding orbit.

Let us now show that $\mathcal{R}$ corresponds to a set of angles which is dense in  $[0,2\pi]$. For this purpose, let us consider the matrix product $A_2 A_1$. One can check that this matrix has complex non real eigenvalues and that it is similar to the rotation
\[
R=\left(
\begin{array}{cc}
\frac{\sqrt{5}}{2\sqrt{2}} & -\frac{\sqrt{7}}{2\sqrt{2}} \\
\frac{\sqrt{7}}{2\sqrt{2}} & \frac{\sqrt{5}}{2\sqrt{2}} 
\end{array}
\right).
\]
The angle $\theta$ of the rotation is such that $\cos2\theta=\frac{1}{4}$, from which one easily deduces that this angle is incommensurable with $\pi$ (see e.g.~\cite{jahnel}). We thus deduce that the set of points obtained applying iteratively $A_2 A_1$ and starting from the $x_1$ axis correspond to a set of angles which is dense in $[0,2\pi]$.
This concludes the proof of the proposition.
\end{proof}
Note that the property that the orbit corresponds to a set of angles dense in $[0,2\pi]$ remains true  if one replaces a point of the $x_1$ axis with any other starting point.

Since the value $\tilde\rho_z(\cM)$ is constant on each orbit, and thus on $\mathcal{R}$, we have the following important consequence of Proposition~\ref{prop-dense}.
\begin{corollary}
There exists a dense subset $\Omega$ of $\re^2$, with $\mathcal{R}\subseteq \Omega$, such that $\tilde\rho_z(\cM)=\frac12$ for each $z\in \Omega$.
\end{corollary}
At the light of the above result, the following question is very natural:
\begin{center}
\textit{Is there a point $z\in \mathbb{R}^2$ (and thus a dense subset of $\mathbb{R}^2$) such that  $\tilde\rho_z(\cM)>\frac12$?}
\end{center}
A positive answer to this question would reveal interesting discontinuity properties of the function $z\mapsto\tilde\rho_z(\cM)$.
However, due to the complexity of the set of trajectories starting from a given point $z$, giving an estimate of the value $\tilde\rho_z(\cM)$ appears to be very difficult. 
Summing up, for the  considered example the issue of  approximating with arbitrary precision the value $\tilde\rho(\cM)$ by simple numerical methods appears to be a very hard, if not intractable, task. It seems that any accurate computation of the stabilization radius, even for very simple examples, should likely rely on \textit{ad hoc} arguments.

\section{The continuous-time case}
\label{s-cont}

Several notions and results in Sections~\ref{sec-properties} and~\ref{sec-control} are still valid in the continuous-time case without need of substantial changes \colp{(recalling that, according to Assumption~\ref{assum}, in this case $\cM$ is a compact and convex set of matrices)}.
In particular the notion of pointwise stabilizabity given in Definition~\ref{def-pointwise} is still meaningful. Similarly, the stabilizability indices introduced in Definition~\ref{ref-def} are still well-defined, and the continuous-time counterparts of Proposition~\ref{equiv} and Proposition~\ref{p-optim} hold true. A partial counterpart of Lemma~\ref{lem-basic} also holds, as it is straightforward to show that $\tilde \rho (\cM+\gamma \mathrm{Id}) = e^\gamma \tilde \rho (\cM)$, where $\cM+\gamma \mathrm{Id}$ is the set of matrices of the form $M+\gamma \mathrm{Id}$ with $M\in\cM$.

On the other hand, in the continuous-time case one has to be careful in the definition of feedback stabilizability in order to avoid existence and uniqueness issues for ordinary differential equations with possibly discontinuous right-hand side. In particular, several non-equivalent and meaningful notions of solutions exist in the case of a discontinuous feedback $\sigma(\cdot)$. To  circumvent this problem, we choose here to deal with an adapted notion of (uniform) feedback stabilizability based on  a ``sample-and-hold'' scheme.  
\colp{Note that the use of this kind of schemes is consistent with the sampling process used in computer control. Moreover the use of other notions of solutions, such as Filippov solutions, appears to be less practical and may prevent to obtain fully general stabilization results (see e.g.~\cite{clarke}). Even so, we think that the study of the feedback stabilization properties for linear switched systems, in relation with other notions of solutions,  is worth further investigation.}


\begin{definition}
\label{def-sample}
For
 $\delta>0$ we say that a $0$-homogeneous function $\sigma:\mathbb{R}^d\to \cM$ is a $\delta$-stabilizing feedback
if for any $x_0\in \mathbb{R}^d$, the trajectory starting from $x_0$ and defined by $\dot{x}(t)=A_{\sigma(x(\delta k))}x(t)$ on each interval $[\delta k, \delta (k+1)),\ k\in\mathbb{N},$  converges to $0$.

In addition we say that $\mathcal{M}$ is uniformly feedback stabilizable if for any positive $\epsilon_1,\epsilon_2$ there exists $T\in \re_+$  such that for any $\delta>0$ small enough there exists a $\delta$-stabilizing feedback such that if $|x_0|<\epsilon_1,$ the 
trajectory $x(\cdot)$  starting from $x_0$ satisfies $|x(t)|<\epsilon_2$ if $t\geq T$. 
\end{definition}
Note that, for the stabilization of continuous-time nonlinear control systems, the use of sample-and-hold feedback controls in combination with  control-Lyapunov functions 
is not new. It has been investigated for instance in~\cite{clarke}. In that paper the asymptotic controllability property for nonlinear control systems is proved under the assumption that there exists a continuous control-Lyapunov function, and allowing more general sampling times  compared to our Definition~\ref{def-sample} (namely, the differences among subsequent sampling times are small enough but not necessarily equal). The following result is the continuous-time counterpart of Proposition~\ref{Lyap}. Here the notion of control-Lyapunov function is that of Definition~\ref{Lyap-def} with the exception that $t\in \mathbb{R}_+$ in Item~$(2)$.
\begin{proposition}\label{Lyap-cont}
If $\mathcal{M}$ admits a Lipschitz continuous control-Lyapunov function $V$ then $\mathcal{M}$ is uniformly feedback stabilizable, and, for each $\delta>0$ small enough, the $\delta$-stabilizing feedback $\sigma(\cdot)$ can be taken piecewise constant. 
\end{proposition}
The previous result may not be seen as a particular case of the main theorem in~\cite{clarke}, as the latter does not guarantee any regularity property of the feedback law.  Also, opposed to~\cite{clarke}, our result can be proved without resorting to sophisticated tools from nonsmooth analysis, as shown below.
\begin{proof}[Proof of Proposition~\ref{Lyap-cont}]
Given an initial point $x_0\in\mathbb{R}^d$ let us consider the switching law $\sigma_{x_0}(\cdot)$ as given by Definition~\ref{Lyap-def}. 
For $\delta>0$ let us consider the matrix $\bar{A}=\frac1\delta \int_0^{\delta} A_{\sigma_{x_0}(s)}ds$. Since $\mathcal{M}$ is convex 
we have that $\bar{A}\in \mathcal{M}$, that is there exists $\bar\sigma(x_0)$ such that $A_{\bar\sigma(x_0)}=\bar{A}$. Consider now the solution $\bar x(\cdot)$ of the equation $\dot x=\bar A x$ with initial condition $x_0$ and set $y(\cdot)=x_{\sigma_{x_0},x_0}(\cdot)-\bar x(\cdot)$. Then $y(\cdot)$ satisfies the equation 
\[\dot y(t)=\bar A y(t)+(A_{\sigma(t)}-\bar A) x_{\sigma_{x_0},x_0}(t)\]
with initial condition $y(0)=0$. The corresponding solution at time $\delta$, given by the variation of constants formula, is
\[y(\delta)=e^{\delta\bar A} \int_0^{\delta} e^{-t\bar A}(A_{\sigma(t)}-\bar A) x_{\sigma_{x_0},x_0}(t)dt.\]
Since 
\begin{align*}
e^{-t\bar A}(A_{\sigma(t)}-\bar A) x_{\sigma_{x_0},x_0}(t)=&(e^{-t\bar A}-\mathrm{Id})(A_{\sigma(t)}-\bar A) x_{\sigma_{x_0},x_0}(t)+\\
&+(A_{\sigma(t)}-\bar A) (x_{\sigma_{x_0},x_0}(t)-x_0)+(A_{\sigma(t)}-\bar A)x_0
\end{align*}
and by compacity of $\mathcal{M}$, and since moreover the integral on $[0,\delta]$ of the last term is zero,  we have that (for $\delta$ small) there exists a constant $C>0$ such that 
\[|y(\delta)|\leq C\delta^2|x_0|.\]
Let $L$ be the Lipschitz constant of the Lyapunov function $V$.
Then 
\begin{align*}
V(x_0)-V(\bar x(\delta)) &\geq  V(x_0)-V(x_{\sigma_{x_0},x_0}(\delta))-|V(\bar x(\delta))-V(x_{\sigma_{x_0},x_0}(\delta))|\\
&\geq  m(1-\mu^{\delta})|x_0|- C\delta^2|x_0|\\
&\geq  \frac12 m\Big(\log\frac1{\mu}\Big)\delta |x_0| 
\end{align*}
for 
$\delta$ small enough, which immediately shows that $\bar \sigma(\cdot)$ is a $\delta$-stabilizing feedback.

Finally, exactly as in the discrete-time case, the fact that the $\delta$-stabilizing feedback can be taken piecewise constant follows from the continuity of $V$ and the continuous dependence of the solutions with respect to the initial data.
\end{proof}
\begin{remark}
The feedback constructed in the proof of Proposition~\ref{Lyap-cont} actually stabilizes the system robustly with respect to small perturbations of the sampling times. 
\end{remark}
Proposition~\ref{Vlam} and Corollary~\ref{cor-feedback-pointwise} still hold in the continuous-time case (in particular a proof of the continuous-time counterpart of Proposition~\ref{Vlam} may be adapted from an analogous result in \cite[Section 4.3.1]{sun}) and they are summarized by the following result. Note that converse Lyapunov theorems for the stabilization (or asymptotic controllability) of more general classes of nonlinear control systems may be found for instance in~\cite{kellett-teel,rifford}.
\begin{proposition}
\label{Vlam-cont}
For any $\lambda>\tilde{\rho}(\cM)$ the function $V_{\lambda}:\mathbb{R}^d\to \mathbb{R}_+$ 
\begin{equation}
\label{lyap}
V_{\lambda}(x)=\sup_{t\geq 0}\inf_{\sigma(\cdot)}\frac{|x_{\sigma,x}(t)|}{\lambda^t}
\end{equation}
is well-defined, absolutely homogeneous, Lipschitz continuous and satisfies $V_{\lambda}(x_{\sigma,x}(t))\leq\lambda^t V_{\lambda}(x)$ for any $x\in\mathbb{R}^d,\ t\in\mathbb{R}_+$ and some $\sigma(\cdot)$, depending on $x$. In particular $\mathcal{M}$ is pointwise stabilizable if and only if it admits an absolutely homogeneous control-Lyapunov function. As a consequence the following conditions are equivalent:
\begin{itemize}
\item[(i)] $\tilde{\rho}'(\cM)<1,$
\item[(ii)] $\tilde{\rho}(\cM)<1,$
\item[(iii)] $\mathcal{M}$ is pointwise stabilizable,
\item[(iv)] $\mathcal{M}$ is uniformly feedback stabilizable.
\end{itemize}
\end{proposition}

\section{Conclusion and open questions}
Stabilizability of switched systems is a natural goal for control theorists, with several promising applications.  The problem raises many subtle mathematical questions. In this paper, we tried to motivate a thorough analysis of these questions.  We showed that many counterintuitive phenomena occur, which one must pay attention to for a proper treatment of the problem.  
\rmj{From a numerical point of view, the stabilization problem appears to be extremely challenging. The discussion and the results obtained throughout the paper raise several open question. Among them we wish to recall the next two ones:}
\begin{opq}
Is there an algorithm, or a characterization, to decide whether $\tilde \rho(\cM) =\check \rho(\cM)?$ 
\end{opq}
\begin{opq}Is it possible to improve Theorem~\ref{lem-sing} and provide a generally better formula for a lower bound? In particular, how to find a better lower bound on $\tilde\rho(\cM)$ in Example \ref{ex-urbano}?
\end{opq}
\rmj{As for the first question, let us mention that sufficient conditions for this equivalence, for a slightly different notion of feedback stabilizability (stronger than ours, motivated by applications involving supervisory control and measurement scheduling), has been shown in~\cite{lee-dul2011}.} Concerning the second question, let us notice that a result improving Theorem~\ref{lem-sing} could for instance lead to a better estimate (from below) of $\tilde\rho(\cM)$ in Example \ref{ex-urbano}. Nevertheless the analysis in Section~\ref{stanford} seems to suggest that an exact computation of such a quantity  should rely on \textit{ad hoc} methods.


\bibliographystyle{plain}

\end{document}